\documentclass[12pt]{article}

\usepackage{amsmath,amsthm,amssymb}

\newcommand{\dashint}{-\!\!\!\!\!\!\displaystyle\int}
\newcommand{\dint}{\displaystyle\int}
\newcommand{\R}{\mathbb{R}}
\newcommand{\X}{\mathbb{X}}
\newcommand{\dd}{\partial}

\newcommand{\smoothone}[2]{K_{#2}{\left(#1\right)}}
\newcommand{\smoothtwo}[2]{K^2_{#2}{\left(#1\right)}}
\newcommand{\e}{\varepsilon}
\renewcommand{\a}{\alpha}
\renewcommand{\b}{\beta}
\renewcommand{\d}{\delta}
\newcommand{\g}{\gamma}
\newcommand{\x}{\xi}
\newcommand{\one}{\textbf{1}}

\theoremstyle{plain}
\newtheorem{lemm}{Lemma}
\newtheorem{thmm}[lemm]{Theorem}
\newtheorem{corr}[lemm]{Corollary}

\numberwithin{equation}{section}
\numberwithin{lemm}{section}

\title{Homogenization in Perforated Domains \\ and Interior Lipschitz Estimates}
\date{}
\author{B. Chase Russell\footnote{Supported in part by NSF grant DMS-1161154.}}

\allowdisplaybreaks

\begin{document}

{\maketitle}

\abstract{We establish interior Lipschitz estimates at the macroscopic scale for solutions to systems of linear elasticity with rapidly oscillating periodic coefficients and mixed boundary conditions in domains periodically perforated at a microscopic scale $\e$ by establishing $H^1$-convergence rates for such solutions.  The interior estimates are derived directly without the use of compactness via an argument presented in~\cite{smart} that was adapted for elliptic equations in~\cite{armstrong} and~\cite{shen}.  As a consequence, we derive a Liouville type estimate for solutions to the systems of linear elasticity in unbounded periodically perforated domains.}

\vspace{5mm}



%%%%%%%%%%%%%%%%%%%%%%%%%%%%%%%%%%%%%%%%%%%%%%%%%%%%%%%%
%
%
%
%
%
\section{Introduction}\label{section1}
%
%
%
%
%
%%%%%%%%%%%%%%%%%%%%%%%%%%%%%%%%%%%%%%%%%%%%%%%%%%%%%%%%

The purpose of this paper is to establish $H^1$-convergence rates in periodic homogenization and to establish interior Lipschitz estimates at the macroscopic scale for solutions to systems of linear elasticity in domains periodically perforated at a microscopic scale $\e$.  To be precise, we consider the operator
\begin{equation}\label{one}
\mathcal{L}_{\e}=-\text{div}\left({A^\e(x)\nabla}\right)=-\dfrac{\dd}{\dd x_i}\left(a_{ij}^{\a\b}\left(\dfrac{x}{\e}\right)\dfrac{\dd}{\dd x_j}\right),\,\,\,x\in\e\omega,\,\e>0,
\end{equation}
where $A^\e(x)=A(x/\e)$, $A(y)=\{a_{ij}^{\a\b}(y)\}_{1\leq i,j,\a,\b\leq d}$ for $y\in \omega$, $d\geq 2$, and $\omega\subseteq\R^d$ is an unbounded Lipschitz domain with 1-periodic structure, i.e., if $\one_+$ denotes the characteristic function of $\omega$, then $\one_+$ is a 1-periodic function in the sense that
\[
\one_+(y)=\one_+(z+y)\,\,\,\text{ for }y\in\R^d,\,z\in\mathbb{Z}^d.
\]
The summation convention is used throughout.  We write $\e\omega$ to denote the $\e$-homothetic set $\{x\in\R^d\,:x/\e\in\omega\}$.  We assume $\omega$ is connected and that any two connected components of $\R^d\backslash\omega$ are separated by some positive distance.  This is stated more precisely in Section~\ref{section2}.  We also assume each connected component of $\R^d\backslash\omega$ is bounded.

We assume the coefficient matrix $A(y)$ is real, measurable, and satisfies the elasticity conditions
\begin{align}
&a_{ij}^{\a\b}(y)=a_{ji}^{\b\a}(y)=a_{\a j}^{i\b}(y), \label{two}\\
&\kappa_1|\xi|^2\leq a_{ij}^{\a\b}(y)\xi_i^\a\xi_j^\b\leq \kappa_2|\xi|^2,\label{three}
\end{align}
for $y\in\omega$ and any symmetric matrix $\xi=\{\xi_i^\a\}_{1\leq i,\a\leq d}$, where $\kappa_1,\kappa_2>0$.  We also assume $A$ is 1-periodic, i.e.,
\begin{equation}\label{fiftyseven}
A(y)=A(y+z)\,\,\,\text{ for }y\in\omega,\,z\in\mathbb{Z}^d.
\end{equation}
The coefficient matrix of the systems of linear elasticity describes the linear relation between the stress and strain a material experiences during relatively small elastic deformations.  Consequently, the elasticity conditions~\eqref{two} and~\eqref{three} should be regarded as physical parameters of the system, whereas $\e$ is clearly a geometric parameter.

For a bounded domain $\Omega\subset\R^d$, we write $\Omega_\e$ to denote the domain $\Omega_\e=\Omega\cap\e\omega$.  In this paper, we consider the mixed boundary value problem given by
\begin{equation}\label{five}
\begin{cases}
\mathcal{L}_\e({u_{\e}})={0}\,\,\,\text{ in }\Omega_\e, \\
\sigma_\e(u_\e)=0\,\,\,\text{ on }S_\e:=\dd\Omega_\e\cap\Omega \\
u_{\e}=f\,\,\,\text{ on }\Gamma_\e:=\dd\Omega_\e\cap\dd\Omega,
\end{cases}
\end{equation}
where $\sigma_\e=-nA^{\e}(x)\nabla$ and $n$ denotes the outward unit normal to $\Omega_\e$.  We say $u_{\e}$ is a weak solution to~\eqref{five} provided
\begin{equation}\label{six}
\dint_{\Omega_\e}a_{ij}^{\a\b\e}\dfrac{\dd u_{\e}^\b}{\dd x_j}\dfrac{\dd w^\a}{\dd x_i}=0,\,\,\,w=\{w^\a\}_\a\in H^1(\Omega_\e,\Gamma_\e;\R^d),
\end{equation}
and $u_{\e}-f\in H^1(\Omega_\e,\Gamma_\e;\R^d)$, where $H^1(\Omega_\e,\Gamma_\e;\R^d)$ denotes the closure in $H^1(\Omega_\e;\R^d)$ of $C^\infty(\R^d;\R^d)$ functions vanishing on $\Gamma_\e$.  The boundary value problem~\eqref{five} models relatively small elastic deformations of composite materials subject to zero external body forces (see~\cite{yellowbook}).

If $\omega=\R^d$---the case when $\Omega_\e=\Omega$---then the existence and uniqueness of a weak solution $u_\e\in H^1(\Omega_\e;\R^d)$ to~\eqref{five} for a given $f\in H^{1}(\Omega;\R^d)$ follows easily from the Lax-Milgram theorem and Korn's first inequality.  If $\omega\subsetneq\R^d$, then the existence and uniqueness of a weak solution to~\eqref{five} still follows from the Lax-Milgram theorem but in addition Korn's first inequality for perforated domains (see Lemma~\ref{fortysix}).

One of the main results of this paper is the following theorem.  For any measurable set $E$ (possibly empty) and ball $B(x_0,r)\subset\R^d$ with $r>0$, denote
\[
\dashint_{B(x_0,r)\cap E}f(x)\,dx=\dfrac{1}{r^d}\int_{B(x_0,r)\cap E}f(x)\,dx
\]

\begin{thmm}\label{nineteen}
Suppose $A$ satisfies~\eqref{two},~\eqref{three}, and~\eqref{fiftyseven}.  Let $u_\e$ denote a weak solution to $\mathcal{L}_\e(u_\e)=0$ in $B(x_0,R)\cap\e\omega$ and $\sigma_\e(u_\e)=0$ for $B(x_0,R)\cap\dd(\e\omega)$ for some $x_0\in\R^d$ and $R>0$.  For $\e\leq r<R/3$, there exists a constant $C$ depending on $d$, $\omega$, $\kappa_1$, and $\kappa_2$ such that
\begin{equation}\label{seventyfour}
\left(\dashint_{B(x_0,r)\cap\e\omega}|\nabla u_\e|^2\right)^{1/2}\leq C\left(\dashint_{B(x_0,R)\cap\e\omega}|\nabla u_\e|^2\right)^{1/2}.
\end{equation}
\end{thmm}

The scale-invariant estimate in Theorem~\ref{nineteen} should be regarded as a Lipschitz estimate for solutions $u_\e$, as under additional smoothness assumptions on the coefficients $A$ we may deduce interior Lipschitz estimate for solutions to~\eqref{five} from local Lipschitz estimates for $\mathcal{L}_1$ and a ``blow-up argument'' (see the proof of Lemma~\ref{thirtyseven}).  In particular, if $A$ is H\"{o}lder continuous, i.e., there exists a $\tau\in (0,1)$ with 
\begin{equation}\label{fiftyfour}
|A(x)-A(y)|\leq C|x-y|^\tau\,\,\,\text{ for }x,y\in\omega
\end{equation}
for some constant $C$ uniform in $x$ and $y$, we may deduce the following corollary.

\begin{corr}\label{fiftyfive}
Suppose $A$ satisfies~\eqref{two},~\eqref{three},~\eqref{fiftyseven}, and~\eqref{fiftyfour}, and suppose $\omega$ is an unbounded $C^{1,\a}$ domain for some $\a>0$.  Let $u_\e$ denote a weak solution to $\mathcal{L}_\e(u_\e)=0$ in $B(x_0,R)\cap\e\omega$ and $\sigma_\e(u_\e)=0$ for $B(x_0,R)\cap\dd(\e\omega)$ for some $x_0\in\R^d$ and $R>0$.  Then
\begin{equation}\label{sixtyone}
\|\nabla u_\e\|_{L^\infty(B(x_0,R/3)\cap\e\omega)}\leq C\left(\dashint_{B(x_0,R)\cap\e\omega}|\nabla u_\e|^2\right)^{1/2},
\end{equation}
where $C$ depends on $d$, $\omega$, $\kappa_1$, $\kappa_2$, $\tau$, and $\a$.
\end{corr}

Another consequence of Theorem~\ref{nineteen} is the following Liouville type property for systems of linear elasticity in unbounded periodically perforated domains.  In particular, we have the following corollary.

\begin{corr}\label{fiftyeight}
Suppose $A$ satisfies~\eqref{two},~\eqref{three}, and~\eqref{fiftyseven}, and suppose $\omega$ is an unbounded Lipschitz domain with 1-periodic structure.  Let $u$ denote a weak solution of $\mathcal{L}_1(u)=0$ in $\omega$ and $\sigma_1(u)=0$ on $\dd\omega$.  Assume
\begin{equation}\label{sixtytwo}
\left(\dashint_{B(0,R)\cap\omega}|u|^2\right)^{1/2}\leq CR^{\nu},
\end{equation}
for some $\nu\in (0,1)$, some constant $C:=C(u)>0$, and for all $R>1$.  Then $u$ is constant.
\end{corr}

Interior Lipschitz estimates for the case $\omega=\R^d$ were first obtained \textit{indirectly} through the method of compactness presented in~\cite{avellaneda}.  Interior Lipschitz estimates for solutions to a single elliptic equation in the case $\omega\subsetneq\R^d$ were obtained indirectly in~\cite{yeh} through the same method of compactness.  The method of compactness is essentially a ``proof by contradiction'' and relies on the qualitative convergence of solutions $u_\e$ (see Theorem~\ref{seventyone}).  The method relies on sequences of operators $\{\mathcal{L}_{\e_k}^k\}_k$ and sequences of functions $\{u_k\}_k$ satisfying $\mathcal{L}_{\e_k}^k(u_k)=0$, where $\mathcal{L}_{\e_k}^k=-\text{div}(A_k^{\e_k}\nabla )$, $\{A_k^{\e_k}\}_k$ satisfies~\eqref{two},~\eqref{three}, and~\eqref{fiftyseven} in $\omega+s_k$ for $s_k\in\R^d$.
In the case $\omega=\R^d$, then $\omega+s_k=\R^d$ for any $s_k\in\R^d$, and so it is clear that estimate~\eqref{seventyfour} is uniform in affine transformations of $\omega$.  In the case $\omega\subsetneq\R^d$, affine shifts of $\omega$ must be considered, which complicates the general scheme.

Interior Lipschitz estimates for the case $\omega=\R^d$ were obtained \textit{directly} in~\cite{shen} through a general scheme for establishing Lipschitz estimates at the macroscopic scale first presented in~\cite{smart} and then modified for second-order elliptic systems in~\cite{armstrong} and~\cite{shen}.  We emphasize that our result is unique in that Theorem~\ref{nineteen} extends estimates presented in~\cite{shen}---i.e., interior Lipschitz estimates for systems of linear elasticity---to the case $\omega\subsetneq\R^d$ while \textit{completely avoiding the use of compactness methods}.

The proof of Theorem~\ref{nineteen} (see Section~\ref{section4}) relies on the quantitative convergence rates of the solutions $u_\e$.  Let $u_0\in H^1(\Omega;\R^d)$ denote the weak solution of the boundary value problem for the homogenized system corresponding to~\eqref{five} (see~\eqref{seven}), and let $\chi=\{\chi_j^\b\}_{1\leq j,\b\leq d}\in H^1_{\text{per}}(\omega;\R^d)$ denote the matrix of correctors (see~\eqref{nine}), where $H^1_{\text{per}}(\omega;\R^d)$ denotes the closure in $H^1(Q\cap\omega;\R^d)$ of the set of 1-periodic $C^\infty(\R^d;\R^d)$ functions and $Q=[-1/2,1/2]^d$.  In the case $\omega\subsetneq\R^d$, the estimate
\[
\|u_\e-u_0-\e\chi^\e\nabla u_0\|_{H^1(\Omega_\e)}\leq C\e^{1/2}\|u_0\|_{H^3(\Omega)}
\]
was proved in~\cite{book2} under the assumption that $\chi_j^\b\in W^{1,\infty}_{\text{per}}(\omega;\R^d)$ for $1\leq j,\b\leq d$, where $W^{1,\infty}_{\text{per}}(\omega;\R^d)$ is defined similarly to $H^1_{\text{per}}(\omega;\R^d)=W^{1,2}_{\text{per}}(\omega;\R^d)$.  However, if it is only assumed that the coefficients $A$ are real, measurable, and satisfy~\eqref{two},~\eqref{three}, and~\eqref{fiftyseven}, then the first-order correctors are not necessarily Lipschitz.  Consequently, the following theorem is another main result of this paper.  Let $K_\e$ denote the smoothing operator at scale $\e$ defined by~\eqref{eighteen}, and let $\eta_\e\in C_0^\infty(\Omega)$ be the cut-off function defined by~\eqref{eleven}.  The use of the smoothing operator $K_\e$ (details are discussed in Section~\ref{section2}) is motivated by work in~\cite{suslina}.
\begin{thmm}\label{ten}
Let $\Omega$ be a bounded Lipschitz domain and $\omega$ be an unbounded Lipschitz domain with 1-periodic structure.  Suppose $A$ is real, measurable, and satisfies~\eqref{two},~\eqref{three}, and~\eqref{fiftyseven}.  Let $u_\e$ denote a weak solution to~\eqref{five}.  There exists a constant $C$ depending on $d$, $\Omega$, $\omega$, $\kappa_1$, and $\kappa_2$ such that
\[
\|u_\e-u_0-\e\chi^\e\smoothtwo{(\nabla u_0)\eta_\e}{\e}\|_{H^1(\Omega_\e)}\leq C\e^{1/2}\|f\|_{H^1(\dd\Omega)}.
\]
\end{thmm}

This paper is structured in the following manner.  In Section~\ref{section2}, we establish notation and recall various preliminary results from other works.  The convergence rate presented in Theorem~\ref{ten} is proved in Section~\ref{section3}.  In Section~\ref{section4}, we prove the interior Lipschitz esitmates given by Theorem~\ref{nineteen} and provide the proof of Corollary~\ref{fiftyfive}.  To finish the section, we prove the Liouville type property Corollary~\ref{fiftyeight}.

%%%%%%%%%%%%%%%%%%%%%%%%%%%%%%%%%%%%%%%%%%%%%%%%%%%%%%%%
%
%
%
%
%
\section{Notation and Preliminaries}\label{section2}
%
%
%
%
%
%%%%%%%%%%%%%%%%%%%%%%%%%%%%%%%%%%%%%%%%%%%%%%%%%%%%%%%

Fix $\zeta\in C_0^\infty (B(0,1))$ so that $\zeta\geq 0$ and $\int_{\R^d}\zeta=1$.  Define
\begin{equation}\label{eighteen}
\smoothone{g}{\e}(x)=\dint_{\R^d}g(x-y)\zeta_\e(y)\,dy,\,\,\,f\in L^2(\R^d)
\end{equation}
where $\zeta_\e(y)=\e^{-d}\zeta(y/\e)$.  Note $K_\e$ is a continuous map from $L^2(\R^d)$ to $L^2(\R^d)$.  A proof for each of the following two lemmas is readily available in~\cite{shen}, and so we do not present either here.  For any function $g$, set $g^\e(\cdot)=g(\cdot/\e)$.

\begin{lemm}\label{sixteen}
Let $g\in H^1(\R^d)$.  Then
\[
\|g-\smoothone{g}{\e}\|_{L^2(\R^d)}\leq C\e\|\nabla g\|_{L^2(\R^d)},
\]
where $C$ depends only on $d$.
\end{lemm}

\begin{lemm}\label{fifteen}
Let $h\in L_{\text{loc}}^2(\R^d)$ be a 1-periodic function.  Then for any $g\in L^2(\R^d)$,
\[
\|h^\e\smoothone{g}{\e}\|_{L^2(\R^d)}\leq C\|h\|_{L^2(Q)}\|g\|_{L^2(\R^d)}
\]
\end{lemm}

A proof of Lemma~\ref{sixtynine} can be found in~\cite{book2}.

\begin{lemm}\label{sixtynine}
Let $\Omega\subset\R^d$ be a bounded Lipschitz domain.  For any $g\in H^1(\Omega)$,
\[
\|g\|_{L^2(\mathcal{O}_r)}\leq Cr^{1/2}\|g\|_{H^1(\Omega)},
\]
where $C$ depends on $d$ and $\Omega$, and $\mathcal{O}_{r}=\{x\in\Omega\,:\,\text{dist}(x,\dd\Omega)<r\}$.
\end{lemm}

A proof of Lemma~\ref{fourteen} can be found in~\cite{yellowbook}.  

\begin{lemm}\label{fourteen}
Suppose $B=\{b_{ij}^{\a\b}\}_{1\leq i,j,\a,\b\leq d}$ is 1-periodic and satisfies $b_{ij}^{\a\b}\in L_{\text{loc}}^2(\R^d)$ with
\[
\dfrac{\dd}{\dd y_i}b_{ij}^{\a\b}=0,\,\,\,\text{ and }\,\,\,\dint_Q b_{ij}^{\a\b}=0.
\]
There exists $\pi=\{\pi_{kij}^{\a\b}\}_{1\leq i,j,k,\a,\b\leq d}$ with $\pi_{kij}^{\a\b}\in H^1_{\text{loc}}(\R^d)$ that is 1-periodic and satisfies
\[
\dfrac{\dd}{\dd y_k}\pi_{kij}^{\a\b}=b_{ij}^{\a\b}\,\,\,\text{ and }\,\,\,\pi_{kij}^{\a\b}=-\pi_{ikj}^{\a\b}.
\]
\end{lemm}

Theorem~\ref{thirteen} is a classical result in the study of periodically perforated domains.  It can be used to prove Korn's first inequality in perforated domains (see Lemma~\ref{fortysix}), which is needed together with the Lax-Milgram theorem to prove the existence and uniqueness of solutions to~\eqref{five}.  For a proof of Theorem~\ref{thirteen}, see~\cite{book2}.

\begin{thmm}\label{thirteen}
Let $\Omega$ and $\Omega_0$ be a bounded Lipschitz domains with $\overline{\Omega}\subset\Omega_0$ and $\text{dist}(\dd\Omega_0,\Omega)>1$.  For $0<\e<1$, there exists a linear extension operator $P_\e: H^1(\Omega_\e,\Gamma_\e;\R^d)\to H_0^1(\Omega_0;\R^d)$ such that 
\begin{align}
&\|P_\e w\|_{H^1(\Omega_0)}\leq C_1\|w\|_{H^1(\Omega_\e)}, \label{thirtyfive}\\
&\|\nabla P_\e w\|_{L^2(\Omega_0)}\leq C_2\|\nabla w\|_{L^2(\Omega_\e)}, \label{thirtyeight}\\
&\|e(P_\e w)\|_{L^2(\Omega_0)}\leq C_3\|e(w)\|_{L^2(\Omega_\e)},
\end{align}
for some constants $C_1$, $C_2$, and $C_3$ depending on $\Omega$ and $\omega$, where $e(w)$ denotes the symmetric part of $\nabla w$, i.e.,
\begin{equation}\label{fortyseven}
e(w)=\dfrac{1}{2}\left[\nabla w+(\nabla w)^T\right].
\end{equation}
\end{thmm}

Korn's inequalities are classical in the study of linear elasticity.  The following lemma is essentially Korn's first inequality but formatted for periodically perforated domains.  Lemma~\ref{fortysix} follows from Theorem~\ref{thirteen} and Korn's first inequality.  For an explicit proof of Lemma~\ref{fortysix}, see~\cite{book2}.

\begin{lemm}\label{fortysix}
There exists a constant $C$ independent of $\e$ such that
\[
\|w\|_{H^1(\Omega_\e)}\leq C\|e(w)\|_{L^2(\Omega_\e)}
\]
for any $w\in H^1(\Omega_\e,\Gamma_\e;\R^d)$, where $e(w)$ is given by~\eqref{fortyseven}.
\end{lemm}

If $\omega=\R^d$, it can be shown that the weak solution to~\eqref{five} converges weakly in $H^1(\Omega;\R^d)$ and consequently strongly in $L^2(\Omega;\R^d)$ as $\e\to 0$ to some $u_0$, which is a solution of a boundary value problem in the domain $\Omega$ (see~\cite{book1} or~\cite{yellowbook}).  Indeed, we have the following known qualitative convergence.

\begin{thmm}\label{seventyone}
Suppose $\omega=\R^d$ and that $\Omega$ is a bounded Lipschitz domain.  Suppose $A$ satisfies ~\eqref{two},~\eqref{three}, and~\eqref{fiftyseven}.  Let $u_\e$ satisfy $\mathcal{L}_\e(u_\e)=0$ in $\Omega$, and $u_\e=f$ on $\dd\Omega$.  Then there exists a $u_0\in H^1(\Omega;\R^d)$ such that
\[
u_\e\rightharpoonup u_0\,\,\,\text{ weakly in }H^1(\Omega;\R^d).
\]
Consequently, $u_\e\to u_0$ strongly in $L^2(\Omega;\R^d)$.
\end{thmm}

For a proof of the previous theorem, see~\cite{book1}, Section 10.3.  The function $u_0$ is called the homogenized solution and the boundary value problem it solves is the homogenized system corresponding to~\eqref{five}.  

If $\omega\subsetneq\R^d$, then it is difficult to qualitatively discuss the convergence of $u_\e$, as $H^1(\Omega_\e;\R^d)$ and $L^2(\Omega_\e;\R^d)$ depend explicitly on $\e$.  Qualitative convergence in this case is discussed in~\cite{siam},~\cite{cioranescu}, and others.  The homogenized system of elasticity corresponding to~\eqref{five} and of which $u_0$ is a solution is given by
\begin{equation}\label{seven}
\begin{cases}
\mathcal{L}_0\left({u_0}\right)={0}\,\,\,\text{ in }\Omega \\
u_0=f\,\,\,\text{ on }\dd\Omega,
\end{cases}
\end{equation}
where $\mathcal{L}_0=-\text{div}(\widehat{A}\nabla)$, $\widehat{A}=\{\widehat{a}_{ij}^{\a\b}\}_{1\leq i,j,\a,\b\leq d}$ denotes a constant matrix given by
\begin{equation}\label{eight}
\widehat{a}_{ij}^{\a\b}=\dashint_{Q\cap\omega}a_{ik}^{\a\g}\dfrac{\dd\X_{j}^{\g\b}}{\dd y_k},
\end{equation}
and $\X_j^\b=\{\X_{j}^{\g\b}\}_{1\leq\g\leq d}$ denotes the weak solution to the boundary value problem
\begin{equation}\label{nine}
\begin{cases}
\mathcal{L}_{1}(\X_j^\b)=0\,\,\,\text{ in }Q\cap\omega \\
\sigma_{1}(\X_j^\b)=0\,\,\,\text{ on }\dd\omega\cap Q \\
\chi_j^\b:=\X_j^\b-y_je^\b\text{ is 1-periodic},\,\,\,\dint_{Q\cap\omega}\chi_j^\b=0,
\end{cases}
\end{equation}
where $e^\b\in\R^d$ has a 1 in the $\b$th position and 0 in the remaining positions.  For details on the existence of solutions to~\eqref{nine}, see~\cite{book2}.  The functions $\chi^{\b}_j$ are referred to as the first-order correctors for the system~\eqref{five}.

It is assumed that any two connected components of $\R^d\backslash\omega$ are separated by some positive distance.  Specifically, if $\R^d\backslash\omega=\cup_{k=1}^\infty H_k$ where $H_k$ is connected and bounded for each $k$, then there exists a constant $\mathfrak{g}^\omega$ so that
\begin{equation}\label{sixtysix}
0< \mathfrak{g}^\omega\leq \underset{i\neq j}{\inf}\left\{\underset{\substack{x_i\in H_i \\ \,x_j\in H_j}}{\inf}|x_i-x_j|\right\}.
\end{equation}

%\[
%\R^d\backslash E=H_0\cup\left(\cup_{j=1}^{N} H_j\right),
%\]
%where $N$ is possibly infinite, $H_j$ is connected for $j=0,1,..., N$, $H_0$ is either unbounded or empty, and $H_j$ is bounded for $j=1,2,..., N$.  Let
%\[
%\mathcal{O}_r^E=\{x\in E\,:\,\text{dist}(x,\dd E)<r\}.
%\]
%We define the gap of the set $E$, denoted by $\mathfrak{g}^E$, by
%\begin{equation}\label{sixtysix}
%\mathfrak{g}^E=\min\left\{2\,\underset{r\geq 0}{\sup}\{\mathcal{O}_r^E\text{ is connected}\},
%\underset{\substack{x_i\in H_i,x_j\in H_j \\ i\neq j}}{\inf}|x_i-x_j|\right\}
%\end{equation}

%%%%%%%%%%%%%%%%%%%%%%%%%%%%%%%%%%%%%%%%%%%%%%%%%%%%%%%%
%
%
%
%
%
\section{Convergence Rates in $H^1(\Omega_\e)$}\label{section3}
%
%
%
%
%
%%%%%%%%%%%%%%%%%%%%%%%%%%%%%%%%%%%%%%%%%%%%%%%%%%%%%%%%

In this section, we establish $H^1(\Omega_\e)$-convergence rates for solutions to~\eqref{five} by proving Theorem~\ref{ten}.  It should be noted that if $A$ satisfies~\eqref{two} and~\eqref{three}, then $\widehat{A}$ defined by~\eqref{eight} satisfies conditions~\eqref{two} and~\eqref{three} but with possibly different constants $\widehat{\kappa}_1$ and $\widehat{\kappa}_2$ depending on $\kappa_1$ and $\kappa_2$.  In particular, we have the following lemma.  For a proof of Lemma~\ref{sixtyeight}, see either~\cite{book1},~\cite{yellowbook}, or~\cite{book2}.

\begin{lemm}\label{sixtyeight}
Suppose $A$ satisfies~\eqref{two},~\eqref{three}, and~\eqref{fiftyseven}.  If $\X_j^\b=\{\X_{j}^{\g\b}\}_\g$ denote the weak solutions to~\eqref{nine}, then $\widehat{A}=\{\widehat{a}_{ij}^{\a\b}\}$ defined by
\[
\widehat{a}_{ij}^{\a\b}=\dint_{Q\cap\omega}a_{ik}^{\a\g}\dfrac{\dd\X_{j}^{\g\b}}{\dd y_k}
\]
satisfies $\widehat{a}_{ij}^{\a\b}=\widehat{a}_{ji}^{\b\a}=\widehat{a}_{\a j}^{i\b}$ and
\begin{align*}
\widehat{\kappa}_1|\x|^2\leq \widehat{a}_{ij}^{\a\b}\x_i^\a\x_j^\b\leq \widehat{\kappa}_2|\x|^2
\end{align*}
for some $\widehat{\kappa}_1,\widehat{\kappa}_2>0$ depending $\kappa_1$ and $\kappa_2$ and any symmetric matrix $\x=\{\x_i^\a\}_{i,\a}$.
\end{lemm}

We assume $A$ satisfies~\eqref{two},~\eqref{three} and~\eqref{fiftyseven}.  We assume $\Omega\subset\R^d$ is a bounded Lipschitz domain and $\omega\subseteq\R^d$ is an unbounded Lipschitz domain with 1-periodic structure such that $\R^d\backslash\omega$ is not connected but each connected component is separated by a positive distance $\mathfrak{g}^\omega$.  We also assume that each connected component of $\R^d\backslash\omega$ is bounded.

Let $K_\e$ be defined as in Section~\ref{section2}.  Let $\eta_\e\in C_0^\infty(\Omega)$ satisfy
\begin{equation}\label{eleven}
\begin{cases}
0\leq \eta_\e(x)\leq 1\,\,\,\text{ for }x\in\Omega, \\
\text{supp}(\eta_\e)\subset \{x\in\Omega\,:\,\text{dist}(x,\dd\Omega)\geq 3\e\}, \\
\eta_\e=1\,\,\,\text{ on }\{x\in\Omega\,:\,\text{dist}(x,\dd\Omega)\geq 4\e\}, \\
|\nabla\eta_\e|\leq C\e^{-1}.
\end{cases}
\end{equation}
If $P_\e$ is the linear extension operator provided by Theorem~\ref{thirteen}, then we write $\widetilde{w}=P_\e w$ for $w\in H^1(\Omega_\e,\Gamma_\e;\R^d)$.  Throughout, $C$ denotes a harmless constant that may change from line to line.

\begin{lemm}\label{twenty}
Let
\[
r_{\e}=u_\e-u_0-\e\chi^\e\smoothtwo{(\nabla u_0)\eta_\e}{\e}.
\]
Then
\begin{align*}
&\dint_{\Omega_\e} A^\e\nabla r_{\e}\cdot\nabla w \\
&\hspace{10mm}= |Q\cap\omega|\dint_{\Omega}\widehat{A}\nabla u_0\cdot\nabla\eta_\e\widetilde{w}-|Q\cap\omega|\dint_{\Omega}(1-\eta_\e)\widehat{A}\nabla u_0\cdot\nabla\widetilde{w} \\
&\hspace{20mm}+\dint_\Omega\left[|Q\cap\omega|\widehat{A}-\one_+^\e A^\e\right]\left[\nabla u_0-\smoothtwo{(\nabla u_0)\eta_\e}{\e}\right]\cdot\nabla \widetilde{w} \\
&\hspace{20mm}+\dint_\Omega\left[|Q\cap\omega|\widehat{A}-\one_+^\e A^\e\nabla\X^\e\right]\smoothtwo{(\nabla u_0)\eta_\e}{\e}\cdot\nabla \widetilde{w} \\
&\hspace{20mm}-\e\dint_{\Omega_\e}A^\e\chi^\e\nabla\smoothtwo{(\nabla u_0)\eta_\e}{\e}\cdot\nabla w
\end{align*}
for any $w\in H^1(\Omega_\e,\Gamma_\e;\R^d)$.
\end{lemm}

\begin{proof}
Since $u_\e$ and $u_0$ solve~\eqref{five} and~\eqref{seven}, respectively,
\[
\dint_{\Omega_\e}A^\e\nabla u_\e\cdot\nabla w=0
\]
and
\[
|Q\cap\omega|\dint_{\Omega}\widehat{A}\nabla u_0\cdot\nabla (\widetilde{w}\eta_\e)=0
\]
for any $w\in H^1(\Omega_\e,\Gamma_\e;\R^d)$.  Hence,
\begin{align*}
&\dint_{\Omega_\e}A^\e\nabla r_\e\cdot\nabla w \\
&\hspace{10mm}= \dint_{\Omega_\e}A^\e\nabla u_\e\cdot\nabla w-\dint_{\Omega_\e}A^\e\nabla u_0\cdot\nabla w \\
&\hspace{20mm}-\dint_{\Omega_\e}A^\e\nabla\left[\e\chi^\e\smoothtwo{(\nabla u_0)\eta_\e}{\e}\right]\cdot\nabla w \\
&\hspace{10mm}= |Q\cap\omega|\dint_{\Omega}\widehat{A}\nabla u_0\cdot\nabla (\widetilde{w}\eta_\e)-\dint_{\Omega_\e}A^\e\nabla u_0\cdot\nabla w \\
&\hspace{20mm}-\dint_{\Omega_\e}A^\e\nabla \chi^\e\smoothtwo{(\nabla u_0)\eta_\e}{\e}\cdot\nabla w \\
&\hspace{20mm}-\e\dint_{\Omega_\e}A^\e\chi^\e\nabla\smoothtwo{(\nabla u_0)\eta_\e}{\e}\cdot\nabla w \\
&\hspace{10mm}= |Q\cap\omega|\dint_\Omega\widehat{A}\nabla u_0\cdot\nabla\eta_\e\widetilde{w}-|Q\cap\omega|\dint_\Omega (1-\eta_\e)\widehat{A}\nabla u_0\cdot\nabla\widetilde{w} \\
&\hspace{20mm}+\dint_{\Omega}\left[|Q\cap\omega|\widehat{A}-\one_+^\e A^\e\right]\left[\nabla u_0-\smoothtwo{(\nabla u_0)\eta_\e}{\e}\right]\cdot\nabla\widetilde{w} \\
&\hspace{20mm}+\dint_{\Omega}\left[|Q\cap\omega|\widehat{A}-\one_+^\e A^\e-\one_+^\e A^\e\nabla\chi^\e\right]\smoothtwo{(\nabla u_0)\eta_\e}{\e}\cdot\nabla\widetilde{w} \\
&\hspace{20mm}-\e\dint_{\Omega_\e}A^\e\chi^\e\nabla\smoothtwo{(\nabla u_0) \eta_\e}{\e}\cdot\nabla w,
\end{align*}
which is the desired equality.
\end{proof}

Lemmas~\ref{sixtyfour} presented below is used in the proof of Lemma~\ref{twentyone}, which establishes a Poincar\'{e} type inequality for the perforated domain.  We use the notation $\Delta(x,r)=B(x,r)\cap\dd \Omega$ to denote a surface ball of $\dd\Omega$.

\begin{lemm}\label{sixtyfour}
For sufficiently small $\e$, there exist $r_0,\rho_0>0$ depending only on $\omega$ such that for any $x\in\dd\Omega$,
\[
\Delta\left(y,\e\rho_0\right)\subset\Delta(x,\e r_0)\text{ and }\overline{\Delta\left(y,\e\rho_0\right)}\subset\Gamma_\e
\]
for some $y\in\Gamma_\e$.
\end{lemm}

\begin{proof}
Write $\R^d\backslash\omega=\cup_{j=1}^\infty H_j$, where each $H_j$ is connected and bounded by assumption (see Section~\ref{section2}).  Since $\omega$ is 1-periodic, there exists a constant $M<\infty$ such that
\[
\underset{j\geq 1}{\sup}\,\{\text{diam}\,H_j\}\leq M.
\]
Take
\begin{equation}\label{seventytwo}
r_0=2\max\left\{\mathfrak{g}^\omega,M\right\},
\end{equation}
where $\mathfrak{g}^\omega$ is defined in Section~\ref{section2}.  Set $\rho_0=\frac{1}{16}\mathfrak{g}^\omega$.  Let
\[
\widetilde{H}_j=\left\{z\in\R^d\,:\,\text{dist}(z,H_j)<\frac{1}{4}\mathfrak{g}^\omega\right\}\text{ for each }j,
\]
and fix $x\in\dd\Omega$.  If $x\in\dd\Omega\backslash(\cup_{j=1}^\infty \e \widetilde{H}_j)$, then take $y=x$.  Indeed, for any $z\in\Delta(y,\e \rho_0)\subset\Delta(x,\e r_0)$ and any positive integer $k$,
\begin{align*}
\text{dist}(z,\e H_k) &\geq \text{dist}(y,\e H_k)-|y-z| \\
&\geq \e\frac{1}{4}\mathfrak{g}^\omega-\e\rho_0 \\
&\geq \e\left\{\frac{1}{4}\mathfrak{g}^\omega-\frac{1}{16}\mathfrak{g}^\omega\right\} \\
&\geq \e\frac{3}{16}\mathfrak{g}^\omega,
\end{align*}
and so $\overline{\Delta(y,\e\rho_0)}\subset\Gamma_\e$.

Suppose $x\in\dd\Omega\cap(\cup_{j=1}^\infty \e\widetilde{H}_j)$.  There exists a positive integer $k$ such that $x\in\e \widetilde{H}_k$.  Moreover, $\e\widetilde{H}_k\subset B(x,\e r_0)$ since for any $z\in\e\widetilde{H}_k$ we have
\begin{align*}
|x-z| &\leq \text{dist}(x,\e H_k)+\text{diam}\,(\e H_k)+\text{dist}(z,\e H_k) \\
&\leq \e\frac{1}{4}\mathfrak{g}^\omega+\e M+\e\frac{1}{4}\mathfrak{g}^\omega \\
&< \e\mathfrak{g}^\omega+\e M \\
&< \e r_0.
\end{align*}
In this case, choose $y\in \e(\widetilde{H}_k\backslash H_k)$ so that $\text{dist}(y,\e H_k)= \e(1/8)\mathfrak{g}^\omega$ and $y\in\dd\Omega$.  Then for any $z\in\Delta(y,\e\rho_0)\subset[\dd\Omega\cap\e(\widetilde{H}_k\backslash H_k)]\subset\Delta(x,\e r_0)$,
\begin{align*}
\text{dist}(z,\e H_k) &\geq \text{dist}(y,\e H_k)-|y-z| \\
&\geq \e\frac{1}{8}\mathfrak{g}^\omega-\e\dfrac{1}{16}\mathfrak{g}^\omega \\
&\geq \e\dfrac{1}{16}\mathfrak{g}^\omega,
\end{align*}
and so $\overline{\Delta(y,\e\rho_0)}\subset\Gamma_\e$.
\end{proof}

%\begin{lemm}\label{sixtyfive}
%Suppose $G=B(\widehat{x}_0,1)\times(0,h)\subseteq\R^{d-1}\times\R$ for some $h>0$ and some $x_0=(\widehat{x}_0,0)\in\R^d$.  If $w\in H^1(G;\R^d)$ vanishes on $B(\widehat{y},\rho)\subset\R^{d-1}$ for some $y=(\widehat{y},0)\in\R^d$ and some $0<\rho\leq 1$, then
%\[
%\|w\|_{L^2(G)}\leq Ch\|\nabla w\|_{L^2(G)}.
%\]
%\end{lemm}

\begin{lemm}\label{twentyone}
For $w\in H^1(\Omega_\e,\Gamma_\e;\R^d)$,
\[
\|\widetilde{w}\|_{L^2(\mathcal{O}_{4\e})}\leq C\e\|\nabla \widetilde{w}\|_{L^2(\Omega)},
\]
where $\mathcal{O}_{4\e}=\{x\in\Omega\,:\,\text{dist}(x,\dd\Omega)<4\e\}$ and $C$ depends on $d$, $\Omega$, and $\omega$.
\end{lemm}

\begin{proof}
We cover $\dd\Omega$ with the surface balls $\Delta(x,\e r_0)$ provided in Lemma~\ref{sixtyfour} and partition the region $\mathcal{O}_{4\e}$.  In particular, let $r_0$ denote the constant given by Lemma~\ref{sixtyfour}, and note $\cup_{x\in\dd\Omega}\Delta(x,\e r_0)$ covers $\dd\Omega$, which is compact.  Then there exists $\{x_i\}_{i=1}^{N}$ with $\dd\Omega\subset\cup_{i=1}^{N}\Delta(x_i,\e r_0)$, where $N=N(\e)$.  Write
\[
\mathcal{O}_{4\e}^{(i)}=\{x\in\Omega\,:\,\text{dist}(x,\Delta_i)<4\e\},\,\,\,\text{where }\Delta_i=\Delta(x_i,\e r_0).
\]
Given that $\Omega$ is a Lipschitz domain, there exists a positive integer $M<\infty$ independent of $\e$ such that $\mathcal{O}_{4\e}^{(i)}\cap\mathcal{O}_{4\e}^{(j)}\neq\emptyset$ for at most $M$ positive integers $j$ different from $i$.

Set $W(x)=\widetilde{w}(\e x)$.  Note for each $1\leq i\leq N$, by Lemma~\ref{sixtyfour} there exists a $y_i\in \mathcal{O}_{4\e}^{(i)}$ such that $\widetilde{w}\equiv 0$ on $\Delta(y_i,\e\rho_0)\subset\Delta_i$.  Hence, by Poincar\'{e}'s inequality (see Theorem 1 in~\cite{meyers}),
\begin{equation}\label{seventythree}
\left(\dint_{\mathcal{O}_{4\e}^{(i)}/\e}|W|^2\right)^{1/2}\leq C\left(\dint_{\mathcal{O}_{4\e}^{(i)}/\e}|\nabla W|^2\right)^{1/2},
\end{equation}
where $C$ depends on $\Omega$, $r_0$, and $\rho_0$ but is independent of $\e$ and $i$.  Specifically,
\[
\dint_{\mathcal{O}_{4\e}}|\widetilde{w}(x)|^2\,dx\leq C\e^2\sum_{i=1}^N\dint_{\mathcal{O}_{4\e}^{(i)}}|\nabla \widetilde{w}(x)|^2\,dx\leq C_1\e^2\dint_{\mathcal{O}_{4\e}}|\nabla \widetilde{w}(x)|^2\,dx
\]
where we've made the change of variables $\e x\mapsto x$ in~\eqref{seventythree} and $C_1$ is a constant depending on $\Omega$, $\omega$, and $M$ but independent of $\e$.

%fix $x_0\in\dd\Omega$, and set $W(x)=w(\e x)$.  By Poincar\'{e}'s inequality (see Theorem 1 in~\cite{meyers}),
%\[
%\left(\dint_{B(x_0, r_0)\cap(\Omega/\e)}|W|^2\right)^{1/2}\leq C\left(\dint_{B(x_0, r_0)\cap(\Omega/\e)}|\nabla W|^2\right)^{1/2},
%\]
%where $C$ depends on $\Omega$ and $\omega$ but is independent of $\e$, and $\Omega/\e=\{x\in\R^d\,:\,\e x\in \Omega\}$.  Specifically,
%\[
%\left(\dint_{B(x_0,\e r_0)\cap\Omega}|w(x)|^2\,dx\right)^{1/2}\leq C\e\left(\dint_{B(x_0,\e r_0)\cap\Omega}|\nabla w(x)|^2\,dx\right)^{1/2},
%\]
%where we've made the change of variables $\e x\mapsto x$.
\end{proof}

\begin{lemm}\label{twentysix}
For $w\in H^1(\Omega_\e,\Gamma_\e;\R^d)$,
\begin{align*}
\left|\dint_{\Omega_\e}A^\e\nabla r_\e\cdot\nabla w\right| &\leq C\left\{\|u_0\|_{L^2(\mathcal{O}_{4\e})}+\|(\nabla u_0)\eta_\e-\smoothone{(\nabla u_0)\eta_\e}{\e}\|_{L^2(\Omega)}\right. \\
&\hspace{30mm}\left.+\e\|\smoothone{(\nabla^2 u_0)\eta_\e}{\e}\|_{L^2(\Omega)}\right\}\|w\|_{H^1(\Omega_\e)}
\end{align*}
\end{lemm}

\begin{proof}
By Lemma~\ref{twenty},
\begin{equation}\label{seventy}
\dint_{\Omega_\e}A^\e\nabla r_\e\cdot\nabla w=I_1+I_2+I_3+I_4+I_5,
\end{equation}
where
\begin{align*}
I_1 &= |Q\cap\omega|\dint_{\Omega}\widehat{A}\nabla u_0\cdot\nabla\eta_\e\widetilde{w},\\
I_2 &= -|Q\cap\omega|\dint_{\Omega}(1-\eta_\e)\widehat{A}\nabla u_0\cdot\nabla\widetilde{w},\\
I_3 &= \dint_\Omega\left[|Q\cap\omega|\widehat{A}-\one_+^\e A^\e\right]\left[\nabla u_0-\smoothtwo{(\nabla u_0)\eta_\e}{\e}\right]\cdot\nabla \widetilde{w},\\
I_4 &= \dint_\Omega\left[|Q\cap\omega|\widehat{A}-\one_+^\e A^\e\nabla\X^\e\right]\smoothtwo{(\nabla u_0)\eta_\e}{\e}\cdot\nabla \widetilde{w},\\
I_5 &= -\e\dint_{\Omega_\e}A^\e\chi^\e\nabla\smoothtwo{(\nabla u_0)\eta_\e}{\e}\cdot\nabla w,
\end{align*}
and $w\in H^1(\Omega_\e,\Gamma_\e;\R^d)$.  According to~\eqref{eleven}, $\text{supp}(\nabla\eta_\e)\subset\mathcal{O}_{4\e}$, where $\mathcal{O}_{4\e}=\{x\in\Omega\,:\,\text{dist}(x,\dd\Omega)<4\e\}$.  Moreover, $|\nabla \eta_\e|\leq C\e^{-1}$.  Hence, Lemma~\ref{twentyone}, Lemma~\ref{sixtyeight}, and~\eqref{eleven} imply
\[
|I_1|\leq C\e^{-1}\dint_{\mathcal{O}_{4\e}}|\nabla u_0\cdot\widetilde{w}|\leq C\|\nabla u_0\|_{L^2(\mathcal{O}_{4\e})}\|\nabla \widetilde{w}\|_{L^2(\Omega)}.
\]
Since $\text{supp}(1-\eta_\e)\subset\mathcal{O}_{4\e}$ and $\eta_\e\leq 1$, Lemma~\ref{sixtynine} and~Lemma~\ref{sixtyeight} imply
\begin{align}
|I_2|&\leq C\dint_{\mathcal{O}_{4\e}}\left|\widehat{A}\nabla u_0\cdot\nabla\widetilde{w}\right|\leq C\|\nabla u_0\|_{L^2(\mathcal{O}_{4\e})}\|\nabla \widetilde{w}\|_{L^2(\Omega)}.\nonumber
\end{align}
By Theorem~\ref{thirteen},
\begin{equation}\label{twentytwo}
|I_1+I_2|\leq C\|\nabla u_0\|_{L^2(\mathcal{O}_{4\e})}\|w\|_{H^1(\Omega_\e)}.
\end{equation}

Again, since $\text{supp}(1-\eta_\e)\subset\mathcal{O}_{4\e}$ (see~\eqref{eleven}),
\begin{align*}
&\|\nabla u_0-\smoothtwo{(\nabla u_0)\eta_\e}{\e}\|_{L^2(\Omega)} \\ 
&\hspace{10mm}\leq \|(1-\eta_\e)\nabla u_0\|_{L^2(\Omega)}+\|(\nabla u_0)\eta_\e-\smoothone{(\nabla u_0)\eta_\e}{\e}\|_{L^2(\Omega)} \\
&\hspace{20mm}+\|\smoothone{(\nabla u_0)\eta_\e-\smoothone{(\nabla u_0)\eta_\e}{\e}}{\e}\|_{L^2(\Omega)} \\
&\hspace{10mm}\leq\|\nabla u_0\|_{L^2(\mathcal{O}_{4\e})}+C\|(\nabla u_0)\eta_\e-\smoothone{(\nabla u_0)\eta_\e}{\e}\|_{L^2(\Omega)}.
\end{align*}
Therefore,
\begin{align}
|I_3|&\leq C\|\nabla u_0-\smoothtwo{(\nabla u_0)\eta_\e}{\e}\|_{L^2(\Omega)}\|w\|_{H^1(\Omega_\e)} \nonumber\\
&\leq C\left\{\|\nabla u_0\|_{L^2(\mathcal{O}_{4\e})} \right. \nonumber\\
&\hspace{10mm}\left.+\|(\nabla u_0)\eta_\e-\smoothone{(\nabla u_0)\eta_\e}{\e}\|_{L^2(\Omega)} \right\}\|w\|_{H^1(\Omega_\e)}.\label{twentythree}
\end{align}

Set $B=|Q\cap\omega|\widehat{A}-\one_+ A\nabla\X$.  By~\eqref{eight} and~\eqref{nine}, $B$ satisfies the assumptions of Lemma~\ref{fourteen}.  Therefore, there exists $\pi=\{\pi_{kij}^{\a\b}\}$ that is 1-periodic with
\[
\dfrac{\dd}{\dd y_k}\pi_{kij}^{\a\b}=b_{ij}^{\a\b}\,\,\,\text{ and }\,\,\,\pi_{kij}^{\a\b}=-\pi_{ikj}^{\a\b},
\]
where
\[
b_{ij}^{\a\b}=|Q\cap\omega|\widehat{a}_{ij}^{\a\b}-\one_+a_{ik}^{\a\g}\dfrac{\dd}{\dd y_k}\X_{j}^{\g\b}.
\]
Moreover, $\|\pi_{ij}^{\a\b}\|_{H^1(Q)}\leq C$ for some constant $C$ depending on $\kappa_1$, $\kappa_2$, and $\omega$.  Hence, integrating by parts gives
\begin{align*}
\dint_{\Omega}b_{ij}^{\a\b\e}\smoothtwo{\dfrac{\dd u_0^\b}{\dd x_j}\eta_\e}{\e}\dfrac{\dd\widetilde{w}^\a}{\dd x_i} &=-\e\dint_{\Omega}\pi_{kij}^{\a\b\e}\dfrac{\dd}{\dd x_k}\left[\smoothtwo{\dfrac{\dd u_0^\b}{\dd x_j}\eta_\e}{\e}\dfrac{\dd\widetilde{w}^\a}{\dd x_i}\right] \\
&=-\e\dint_{\Omega}\pi_{kij}^{\a\b\e}\dfrac{\dd}{\dd x_k}\left[\smoothtwo{\dfrac{\dd u_0^\b}{\dd x_j}\eta_\e}{\e}\right]\dfrac{\dd\widetilde{w}^\a}{\dd x_i} ,
\end{align*}
since
\[
\dint_{\Omega}\pi_{kij}^{\a\b\e}\smoothtwo{\dfrac{\dd u_0^\b}{\dd x_j}\eta_\e}{\e}\dfrac{\dd^2\widetilde{w}^\a}{\dd x_k\dd x_i}=0
\]
due to the anit-symmetry of $\pi$.  Thus, by Lemma~\ref{fifteen}, and~\eqref{eleven},
\begin{align}
|I_4|&\leq C\e\|\pi^\e\nabla\smoothtwo{(\nabla u_0)\eta_\e}{\e}\|_{L^2(\Omega)}\|w\|_{H^1(\Omega_\e)} \nonumber\\
&\leq C\left\{\|\nabla u_0\|_{L^2(\mathcal{O}_{4\e})}+\e\|\smoothone{(\nabla^2 u_0)\eta_\e}{\e}\|_{L^2(\Omega)}\right\}\|w\|_{H^1(\Omega_\e)}.\label{twentyfour}
\end{align}

Finally, by Lemma~\ref{fifteen}, and~\eqref{eleven},
\begin{align}
|I_5|\leq C\left\{\|\nabla u_0\|_{L^2(\mathcal{O}_{4\e})}+\e\|\smoothone{(\nabla^2 u_0)\eta_\e}{\e}\|_{L^2(\Omega)}\right\}\|w\|_{H^1(\Omega_\e)}\label{twentyfive}
\end{align}
The desired estimate follows from~\eqref{seventy},~\eqref{twentytwo},~\eqref{twentythree},~\eqref{twentyfour}, and~\eqref{twentyfive}.
\end{proof}

\begin{lemm}\label{twentyseven}
For $w\in H^1(\Omega_\e,\Gamma_\e;\R^d)$,
\[
\left|\dint_{\Omega_\e}A^\e\nabla r_\e\cdot\nabla w\right|\leq C\e^{1/2}\|f\|_{H^1(\dd\Omega)}\|w\|_{H^1(\Omega_\e)} 
\]
\end{lemm}

\begin{proof}
Recall that $u_0$ satisfies $\mathcal{L}_0(u_0)=0$ in $\Omega$, and so it follows from estimates for solutions in Lipschitz domains for constant-coefficient equations that
\begin{equation}\label{twentynine}
\|(\nabla u_0)^*\|_{L^2(\dd\Omega)}\leq C\|f\|_{H^1(\dd\Omega)},
\end{equation}
where $(\nabla u_0)^*$ denotes the nontangential maximal function of $\nabla u_0$ (see~\cite{dahlberg}).  By the coarea formula,
\begin{equation}\label{thirty}
\|\nabla u_0\|_{L^2(\mathcal{O}_{4\e})}\leq C\e^{1/2}\|(\nabla u_0)^*\|_{L^2(\dd\Omega)}\leq C\e^{1/2}\|f\|_{H^1(\dd\Omega)}.
\end{equation}

Notice that if $u_0$ solves~\eqref{seven}, then $\mathcal{L}_0(\nabla u_0)=0$ in $\Omega$, and so we may use the interior estimate for $\mathcal{L}_0$.  That is,
\begin{equation}\label{twentyeight}
|\nabla^2 u_0(x)|\leq\dfrac{C}{\d(x)}\left(\dashint_{B(x,\d(x)/8)}|\nabla u_0|^2\right)^{1/2},
\end{equation}
where $\d(x)=\text{dist}(x,\dd\Omega)$.  In particular,
\begin{align}
\|(\nabla^2 u_0)\eta_\e\|_{L^2(\Omega)} &\leq \left(\dint_{\Omega\backslash\mathcal{O}_{3\e}}|\nabla^2 u_0|^2\right)^{1/2} \nonumber\\
&\leq C\left(\dint_{\Omega\backslash\mathcal{O}_{3\e}}\dashint_{B(x,\d(x)/8)}\left|\dfrac{\nabla u_0(y)}{\d(x)}\right|^2\,dy\>dx\right)^{1/2} \nonumber\\
&\leq C\left(\dint_{3\e}^{C_0}t^{-2}\dint_{\dd\mathcal{O}_t\cap\Omega}\dashint_{B(x,t/8)}|\nabla u_0(y)|^2\,dy \>dS(x)\>dt\right)^{1/2} \nonumber\\
&\hspace{30mm}+C_1\left(\dint_{\Omega\backslash\mathcal{O}_{C_0}}|\nabla u_0|^2\right)^{1/2}\nonumber\\
&\leq C\|(\nabla u_0)^*\|_{L^2(\dd\Omega)}\left(\dint_{3\e}^{C_0}t^{-2}\,dt\right)^{1/2}+C_1\|\nabla u_0\|_{L^2(\Omega)} \nonumber\\
&\leq C\left\{\e^{-1/2}\|f\|_{H^1(\dd\Omega)}+\|f\|_{H^{1/2}(\dd\Omega)}\right\} \nonumber\\
&\leq C\e^{-1/2}\|f\|_{H^1(\dd\Omega)}\label{thirtyone}.
\end{align}
where $C_0$ is a constant depending on $\Omega$, and we've used~\eqref{eleven},~\eqref{twentyeight}, the coarea formula, energy estimates, and~\eqref{twentynine}.  Hence,
\begin{equation}\label{thirtytwo}
\e\|\smoothone{(\nabla^2 u_0)\eta_\e}{\e}\|_{L^2(\Omega)}\leq C\e^{1/2}\|f\|_{H^1(\dd\Omega)}.
\end{equation}

Finally, by Lemma~\ref{sixteen},
\begin{equation}\label{thirtythree}
\|(\nabla u_0)\eta_\e-\smoothone{(\nabla u_0)\eta_\e}{\e}\|_{L^2(\Omega)}\leq C\e^{1/2}\|f\|_{H^1(\dd\Omega)}.
\end{equation}
where the last inequality follows from~\eqref{eleven}, Lemma~\ref{sixteen}, and~\eqref{thirtyone}.  Equations \eqref{thirty}, \eqref{thirtytwo}, and~\eqref{thirtythree} together with Lemma~\ref{twentysix} give the desired estimate.
\end{proof}

\begin{proof}[Proof of Theorem~\ref{ten}]
Note $r_\e\in H^1(\Omega_\e,\Gamma_\e;\R^d)$, and so by Lemma~\ref{twentyseven} and~\eqref{three},
\begin{align*}
\|e(r_\e)\|^2_{L^2(\Omega_\e)} &\leq C\dint_{\Omega_\e}A^\e\nabla r_\e\cdot\nabla r_\e \\
&\leq C\e^{1/2}\|f\|_{H^1(\dd\Omega)}\|r_\e\|_{H^1(\Omega_\e)}.
\end{align*}
Lemma~\ref{fortysix} gives the desired estimate.
\end{proof}

%%%%%%%%%%%%%%%%%%%%%%%%%%%%%%%%%%%%%%%%%%%%%%%%%%%%%%%%
%
%
%
%
%
\section{Interior Lipschitz Estimate}\label{section4}
%
%
%
%
%
%%%%%%%%%%%%%%%%%%%%%%%%%%%%%%%%%%%%%%%%%%%%%%%%%%%%%%%%

In this section, we use Theorem~\ref{ten} to investigate interior Lipschitz estimates down to the scale $\e$.  In particular, we prove Theorem~\ref{nineteen}.  The proof of Theorem~\ref{nineteen} is based on the scheme used in~\cite{shen} to prove boundary Lipschitz estimates for solutions to~\eqref{five} in the case $\omega=\R^d$, which in turn is based on a more general scheme for establishing Lipschitz estimates presented in~\cite{smart} and adapted in~\cite{shen} and~\cite{armstrong}.

The following Lemma is essentially Cacciopoli's inequality in a perforated ball.  The proof is similar to a proof of the classical Cacciopoli's ineqaulity, but nevertheless we present a proof for completeness.

Throughout this section, let $B_\e(r)$ denote the perforated ball of radius $r$ centered at some $x_0\in\R^d$, i.e., $B_\e(r)=B(x_0,r)\cap\e\omega$.  Let $S_\e(r)=\dd(\e\omega)\cap B(x_0,r)$ and $\Gamma_\e(r)=\e\omega\cap\dd B(x_0,r)$.

\begin{lemm}\label{thirtysix}
Suppose $\mathcal{L}_\e(u_\e)=0$ in $B_\e(2)$ and $\sigma_\e(u_\e)=0$ on $S_\e(2)$.  There exists a constant $C$ depending on $\kappa_1$ and $\kappa_2$ such that
\[
\left(\dashint_{B_\e(1)}|\nabla u_\e|^2\right)^{1/2}\leq C\underset{q\in\R^d}{\inf}\left(\dashint_{B_\e(2)}|u_\e-q|^2\right)^{1/2}
\]
\end{lemm}

\begin{proof}
Let $\varphi\in C_0^\infty(B(2))$ satisfy $0\leq\varphi\leq 1$, $\varphi\equiv 1$ on $B(1)$, $|\nabla\varphi|\leq C_1$ for some constant $C_1$.  Let $q\in\R^d$, and set $w=(u_\e-q)\varphi^2$.  By~\eqref{nineteen} and H\"{o}lder's inequality,
\begin{align}
0&=\dint_{B_\e(2)}A^\e\nabla u_\e\nabla w \nonumber\\
&\geq C_2\dint_{B_\e(2)}|e(u_\e)|^2\varphi^2-C_3\dint_{B_\e(2)}|\nabla\varphi|^2|u_\e-q|^2\label{twelve}
\end{align}
for some constants $C_2$ and $C_3$ depending on $\kappa_1$ and $\kappa_2$.  In particular,
\[
\dint_{B_\e(2)}|e(u_\e\varphi)|^2\leq C\dint_{B_\e(2)}|\nabla\varphi|^2|u_\e-q|^2,
\]
where $C$ only depends on $\kappa_1$ and $\kappa_2$.  Since $\varphi\equiv 1$ in $B(1)$ and $u_\e\varphi\in H^1(B_\e(2),\Gamma_\e(2);\R^d)$, equation~\eqref{twelve} together with Lemma~\ref{fortysix} gives the desired estimate.
\end{proof}

We extend Lemma~\ref{thirtysix} to hold for a ball $B_\e(r)$ with $r>0$ by a convenient scaling technique---the so called ``blow-up argument''---often used in the study of homogenization.

\begin{lemm}\label{thirtyseven}
Suppose $\mathcal{L}_\e(u_\e)=0$ in $B_\e(2r)$ and $\sigma_\e(u_\e)=0$ on $S_\e(2r)$.  There exists a constant $C$ depending on $\kappa_1$ and $\kappa_2$ such that
\[
\left(\dashint_{B_\e(r)}|\nabla u_\e|^2\right)^{1/2}\leq \dfrac{C}{r}\underset{q\in\R^d}{\inf}\left(\dashint_{B_\e(2r)}|u_\e-q|^2\right)^{1/2}
\]
\end{lemm}

\begin{proof}
Let $U_\e(x)=u_\e(rx)$, and note $U_\e$ satisfies $\mathcal{L}_{\e/r}(U_\e)=0$ in $B_\e(2)$ and $\sigma_{\e/r}(U_\e)=0$ on $S_\e(2)$.  By Lemma~\ref{thirtysix},
\[
\left(\dashint_{B_{\e/r}(1)}|\nabla U_\e|^2\right)^{1/2}\leq C\underset{q\in\R^d}{\inf}\left(\dashint_{B_{\e/r}(2)}|U_\e-q|^2\right)^{1/2}
\]
for some $C$ independent of $\e$ and $r$.  Note $\nabla U_\e=r\nabla u_\e$, and so
\[
r^{1-d/2}\left(\dashint_{B_{\e}(r)}|\nabla u_\e|^2\right)^{1/2}\leq Cr^{-d/2}\underset{q\in\R^d}{\inf}\left(\dashint_{B_{\e}(2r)}|u_\e-q|^2\right)^{1/2},
\]
where we've made the substitution $rx\mapsto x$.  The desired inequality follows.
\end{proof}

The following lemma is a key estimate in the proof of Theorem~\ref{nineteen}.  Intrinsically, the following Lemma uses the convergence rate in Theorem~\ref{ten} to approximate the solution $u_\e$ with a ``nice'' function.  

\begin{lemm}\label{forty}
Suppose $\mathcal{L}_\e(u_\e)=0$ in $B_\e(3r)$ and $\sigma_\e(u_\e)=0$ on $S_\e(3r)$.  There exists a $v\in H^1(B(r);\R^d)$ with $\mathcal{L}_0(v)=0$ in $B(r)$ and
\[
\left(\dashint_{B_\e(r)}|u_\e-v|\right)^{1/2}\leq C\left(\dfrac{\e}{r}\right)^{1/2}\left(\dashint_{B_\e(3r)}|u_\e|^2\right)^{1/2}
\]
for some constant $C$ depending on $d$, $\omega$, $\kappa_1$, and $\kappa_2$
\end{lemm}

\begin{proof}
With rescaling (see the proof of Lemma~\ref{thirtyseven}), we may assume $r=1$.  By Lemma~\ref{thirtyseven} and estimate~\eqref{thirtyeight} of Lemma~\ref{thirteen},
\[
\left(\dashint_{B(3/2)}|\widetilde{u}_\e|^2\right)^{1/2}+\left(\dashint_{B(3/2)}|\nabla \widetilde{u}_\e|^2\right)^{1/2}\leq C\left(\dashint_{B_\e(3)}|u_\e|^2\right)^{1/2},
\]
where $\widetilde{u}_\e=P_\e u_\e\in H^1(B(3);\R^d)$ and $P_\e$ is the linear extension operator provided in Lemma~\ref{thirteen}.  The coarea formula then implies there exists a $t\in [1,3/2]$ such that
\begin{equation}\label{thirtynine}
\|\nabla \widetilde{u}_\e\|_{L^2(\dd B(t))}+\|\widetilde{u}_\e\|_{L^2(\dd B(t))}\leq C\|u_\e\|_{L^2(B_\e(3))}.
\end{equation}
Let $v$ denote the solution to the Dirichlet problem $\mathcal{L}_0(v)=0$ in $B(t)$ and $v=\widetilde{u}_\e$ on $\dd B(t)$.  Note that $v=u_\e=\widetilde{u}_\e$ on $\Gamma_\e(t)$.  By Theorem~\ref{ten},
\[
\|u_\e-v\|_{L^2(B_\e(t))}\leq C\e^{1/2}\|\widetilde{u}_\e\|_{H^1(\dd B(t))}
\]
since
\[
\|\chi^\e\smoothtwo{(\nabla v)\eta_\e}{\e}\|_{L^2(B_\e(t))}\leq C\|\nabla v\|_{L^2(B(t))},
\]
where we've used notation consistent with Theorem~\ref{ten}.  Hence,~\eqref{thirtynine} gives
\[
\|u_\e-v\|_{L^2(B_\e(1))}\leq \|u_\e-v\|_{L^2(B_\e(t))}\leq C\e^{1/2}\|u_\e\|_{L^2(B_\e(3))}.
\]
\end{proof}

\begin{lemm}\label{fortytwo}
Suppose $\mathcal{L}_0(v)=0$ in $B(2r)$.  For $r\geq \e$, there exists a constant $C$ depending on $\omega,\kappa_1,\kappa_2$ and $d$ such that
\begin{equation}\label{fortythree}
\left(\dashint_{B(r)}|v|^2\right)^{1/2}\leq C\left(\dashint_{B_\e(2r)}|v|^2\right)^{1/2}
\end{equation}
\end{lemm}

\begin{proof}
Let 
\[
T_\e=\{z\in\mathbb{Z}^d\,:\,\e(Q+z)\cap B(r)\neq\emptyset\},
\]
and fix $z\in T_\e$.  Let $\{H_{k}\}_{k=1}^N$ denote the bounded, connected components of $\R^d\backslash\omega $ with $H_k\cap (Q+z)\neq \emptyset$.  Define $\varphi_k\in C_0^\infty(Q^*(z))$ by
\[
\begin{cases}
\varphi_k(x)=1,\,\,\,\text{ if }x\in H_k, \\
\varphi_k(x)=0,\,\,\,\text{ if }\text{dist}(x,H_k)>\frac{1}{4}\mathfrak{g}^\omega, \\
|\nabla\varphi_k|\leq C,
\end{cases}
\]
where $C$ depends on $\omega$, $\mathfrak{g}^\omega>0$ is defined in Section 2 by~\eqref{sixtysix}, and
\[
Q^*(z)=\bigcup_{j=1}^{3^d} (Q+z_j),\,\,\,z_j\in\mathbb{Z}^d\text{ and }|z-z_j|\leq \sqrt{d}.
\]
Set $\varphi=\sum_{k=1}^N\varphi_k\in C_0^\infty(Q^*)$, where $Q^*=Q^*(z)$.  Note by construction $\varphi\equiv 1$ in $Q^*\backslash\omega$.  

Set $V(x)=v(\e x)$.  Note $\mathcal{L}_0(V)=0$ in $Q+z$.  By Poincar\'{e}'s and Cacciopoli's inequalities,
\[
\dint_{(Q+z)\backslash\omega}|V|^2\leq\sum_{k=1}^{N}\dint_{H_k}|V|^2\leq C\dint_{Q^*}|\nabla (V\varphi)|^2\leq C\dint_{Q^*}|V|^2|\nabla\varphi|^2,
\]
where $C$ depends on $\omega$, $\kappa_1$, $\kappa_2$, and $d$ but is independent of $z$.  Specifically, since $\nabla\varphi=0$ in $Q^*\backslash\omega$ and $(Q+z)\subset Q^*$,
\[
\dint_{(Q+z)\cap\omega}|V|^2+\dint_{(Q+z)\backslash\omega}|V|^2\leq C\dint_{Q^*\cap\omega}|V|^2,
\]
where $C$ only depends on $\omega$, $\kappa_1$, $\kappa_2$, and $d$.  Making the change of variables $\e x\mapsto x$ gives
\[
\dint_{\e(Q+z)}|v|^2\leq C\dint_{\e (Q^*\cap\omega)}|v|^2.
\]
Summing over all $z\in T_\e$ gives the desired inequality, since there is a constant $M<\infty$ depending only on $d$ such that $Q^*(z_1)\cap Q^*(z_2)\neq\emptyset$ for at most $M$ coordinates $z_2\in\mathbb{Z}^d$ different from $z_1$.
\end{proof}

For $w\in L^2(B_\e(r);\R^d)$ and $\e,r>0$, set
\begin{equation}\label{fortyfive}
H_\e(r;w)=\dfrac{1}{r}\underset{\substack{M\in\R^{d\times d} \\ q\in\R^d}}{\inf}\left(\dashint_{B_\e(r)}|w-Mx-q|^2\right)^{1/2},
\end{equation}
and set
\[
H_0(r;w)=\dfrac{1}{r}\underset{\substack{M\in\R^{d\times d} \\ q\in\R^d}}{\inf}\left(\dashint_{B(r)}|w-Mx-q|^2\right)^{1/2}.
\]

\begin{lemm}\label{fortyfour}
Let $v$ be a solution of $\mathcal{L}_0(v)=0$ in $B(r)$.  For $r\geq \e$, there exists a $\theta\in (0,1/4)$ such that
\[
H_\e(\theta r;v)\leq \dfrac{1}{2} H_\e(r;v).
\]
\end{lemm}

\begin{proof}
There exists a constant $C_1$ depending on $d$ such that 
\[
H_\e(r;v)\leq C_1H_0(r;v)
\]
for any $r>0$.  It follows from interior $C^2$-estimates for elasticity systems with constant coefficients that there exists $\theta\in (0,1/4)$ with
\[
H_0(\theta r;v)\leq \dfrac{1}{2C_2} H_0(r/2;v),
\]
where $C_2=C_3C_1$ and $C_3$ is the constant in~\eqref{fortythree} given in Lemma~\ref{fortytwo}.  By Lemma~\ref{fortytwo}, we have the desired inequality.
\end{proof}

\begin{lemm}\label{fifty}
Suppose $\mathcal{L}_\e(u_\e)=0$ in $B_\e(2r)$ and $\sigma_\e(u_\e)=0$ on $S_\e(2r)$.  For $r\geq \e$,
\[
H_\e(\theta r; u_\e)\leq \dfrac{1}{2} H_\e(r;u_\e)+\dfrac{C}{r}\left(\dfrac{\e}{r}\right)^{1/2}\underset{q\in\R^d}{\inf}\left(\dashint_{B_\e(3r)}|u_\e-q|^2\right)^{1/2}
\]
\end{lemm}

\begin{proof}
With $r$ fixed, let $v_r\equiv v$ denote the function guaranteed in Lemma~\ref{forty}.  Observe then
\begin{align*}
H_\e(\theta r; u_\e) &\leq \dfrac{1}{\theta r}\left(\dashint_{B_\e(\theta r)}|u_\e-v|^2\right)^{1/2}+H_\e(\theta r;v) \\
&\leq \dfrac{C}{r}\left(\dashint_{B_\e( r)}|u_\e-v|^2\right)^{1/2}+\dfrac{1}{2}H_\e(r;v) \\
&\leq \dfrac{C}{r}\left(\dashint_{B_\e( r)}|u_\e-v|^2\right)^{1/2}+\dfrac{1}{2}H_\e(r;u_\e), \\
\end{align*}
where we've used Lemma~\ref{fortyfour}.  By Lemma~\ref{forty}, we have
\[
H_\e(\theta r; u_\e)\leq \dfrac{C}{r}\left(\dfrac{\e}{r}\right)^{1/2}\left(\dashint_{B_\e(3r)}|u_\e|^2\right)^{1/2}+\dfrac{1}{2}H_\e(r;u_\e).
\]
Since $H$ remains invariant if we subtract a constant from $u_\e$, the desired inequality follows.
\end{proof}

\begin{lemm}\label{fortynine}
Let $H(r)$ and $h(r)$ be two nonnegative continous functions on the interval $(0,1]$.  Let $0<\e<1/6$.  Suppose that there exists a constant $C_0$ with
\[
\begin{cases}
\underset{r\leq t\leq 3r}{\max} H(t)\leq C_0 H(3r), \\
\underset{r\leq t,s\leq 3r}{\max} |h(t)-h(s)|\leq C_0H(3r), \\
\end{cases}
\]
for any $r\in [\e,1/3]$.  We further assume
\[
H(\theta r)\leq \dfrac{1}{2}H(r)+C_0\left(\dfrac{\e}{r}\right)^{1/2}\left\{H(3r)+h(3r)\right\}
\]
for any $r\in [\e,1/3]$, where $\theta\in (0,1/4)$.  Then
\[
\underset{\e\leq r\leq 1}{\max}\left\{H(r)+h(r)\right\}\leq C\{H(1)+h(1)\},
\]
where $C$ depends on $C_0$ and $\theta$.
\end{lemm}

\begin{proof}
See~\cite{shen}.
\end{proof}

\begin{proof}[Proof of Theorem~\ref{nineteen}]
By rescaling, we may assume $R=1$.  We assume $\e\in(0,1/6)$, and we let $H(r)\equiv H_\e(r;u_\e)$, where $H_\e(r;u_\e)$ is defined above by~\eqref{fortyfive}.  Let $h(r)=|M_r|$, where $M_r\in\R^{d\times d}$ satisfies
\[
H(r)=\dfrac{1}{r}\underset{q\in\R^d}{\inf}\left(\dashint_{B_\e(r)}|u_\e-M_rx-q|^2\right)^{1/2}.
\]
Note there exists a constant $C$ independent of $r$ so that
\begin{equation}\label{fortyeight}
H(t)\leq C H(3r),\,\,\,t\in [r,3r].
\end{equation}
Suppose $s,t\in [r,3r]$.  We have
\begin{align}
|h(t)-h(s)| &\leq \dfrac{C}{r}\underset{q\in\R^d}{\inf}\left(\dashint_{B_\e(r)}|(M_t-M_s)x-q|^2\right)^{1/2} \nonumber\\
&\leq \dfrac{C}{t}\underset{q\in\R^d}{\inf}\left(\dashint_{B_\e(t)}|u_\e-M_tx-q|^2\right)^{1/2} \nonumber\\
&\hspace{10mm}+\dfrac{C}{s}\underset{q\in\R^d}{\inf}\left(\dashint_{B_\e(s)}|u_\e-M_sx-q|^2\right)^{1/2} \nonumber\\
&\leq C H(3r),\nonumber
\end{align}
where we've used~\eqref{fortyeight} for the last inequality.  Specifically,
\begin{equation}\label{fiftyone}
\underset{r\leq t,s\leq 3r}{\max}|h(t)-h(s)|\leq CH(3r).
\end{equation}
Clearly
\[
\dfrac{1}{r}\underset{q\in\R^d}{\inf}\left(\dashint_{B_\e(3r)}|u_\e-q|^2\right)^{1/2}\leq H(3r)+h(3r),
\]
and so Lemma~\ref{fifty} implies
\begin{equation}\label{fiftytwo}
H(\theta r)\leq \dfrac{1}{2}H(r)+C\left(\dfrac{\e}{r}\right)^{1/2}\left\{H(3r)+h(3r)\right\}
\end{equation}
for any $r\in [\e,1/3]$ and some $\theta\in (0,1/4)$.  Note equations~\eqref{fortyeight},~\eqref{fifty}, and~\eqref{fiftytwo} show that $H(r)$ and $h(r)$ satisfy the assumptions of Lemma~\ref{fortynine}.  Consequently, 
\begin{align}
\left(\dashint_{B_\e(r)}|\nabla u_\e|^2\right)^{1/2}&\leq \dfrac{C}{r}\underset{q\in\R^d}{\inf}\left(\dashint_{B_\e(3r)}|u_\e-q|^2\right)^{1/2} \nonumber\\
&\leq C\left\{H(3r)+h(3r)\right\} \nonumber\\
&\leq C\left\{H(1)+h(1)\right\} \nonumber\\
&\leq C\left(\dashint_{B_\e(1)}|u_\e|^2\right)^{1/2}.\label{fiftythree}
\end{align}
Since~\eqref{fiftythree} remains invariant if we subtract a constant from $u_\e$, the desired estimate in Theorem~\ref{nineteen} follows.
\end{proof}

\begin{proof}[Proof of Corollary~\ref{fiftyfive}]
Under the H\"{o}lder continuous condition~\eqref{fiftyfour} and the assumption that $\omega$ is an unbounded $C^{1,\a}$ domain for some $\alpha>0$, solutions to the systems of linear elasticity are known to be locally Lipschitz.  That is, if $\mathcal{L}_1(u)=0$ in $B(y,1)\cap\omega$ and $\sigma_1(u)=0$ on $B(y,1)\cap \dd\omega$, then
\begin{equation}\label{fiftynine}
\|\nabla u\|_{L^\infty(B(y,1/3)\cap\omega)}\leq C\left(\dashint_{B(y,1)\cap\omega}|\nabla u|^2\right)^{1/2},
\end{equation}
where $C$ depends on $d$, $\kappa_1$, $\kappa_2$, and $\omega$.

By rescaling, we may assume $R=1$.  To prove the desired estimate, assume $\e\in (0,1/6)$.  Indeed, if $\e\geq 1/6$, then~\eqref{sixtyone} follows from~\eqref{fiftynine}.  From~\eqref{fiftynine}, a ``blow-up argument'' (see the proof of Lemma~\ref{thirtysix}), and Theorem~\ref{nineteen} we deduce
\begin{align*}
\|\nabla u_\e\|_{L^\infty(B(y,\e)\cap\e\omega)}&\leq C\left(\dashint_{B(y,3\e)\cap\e\omega}|\nabla u_\e|^2\right)^{1/2} \\
&\leq C\left(\dashint_{B(x_0,1)\cap\e\omega}|\nabla u_\e|^2\right)^{1/2}
\end{align*}
for any $y\in B(x_0,1/3)$.  The deisred esitmate readily follows by covering $B(x_0,1/3)$ with balls $B(y,\e)$.
\end{proof}

\begin{proof}[Proof of Corollary~\ref{fiftyeight}]
If $u$ satisfies the growth condition~\eqref{sixtytwo}, then by Lemma~\ref{thirtyseven} and Theorem~\ref{nineteen},
\[
\left(\dashint_{B(x_0,r)\cap\omega}|\nabla u|^2\right)^{1/2}\leq C\left(\dashint_{B(x_0,R)\cap\omega}|\nabla u|^2\right)^{1/2}\leq CR^{\nu-1},
\]
where $C$ is independent of $R$.  Take $R\to\infty$ and note $\nabla u=0$ for arbitrarily large $r$.  Since $\omega$ is connected, we conclude $u$ is constant.
\end{proof}

\bibliography{/Users/bchaserussell/Documents/LaTeX/Research/PostQualifying/Paper10272016bibliography}{}
\bibliographystyle{plain}

\vspace{5mm}

\text{Brandon Chase Russell}

\textsc{Department of Mathematics}

\textsc{University of Kentucky}

\textsc{Lexington, KY 40506, USA}

\text{E-mail:} brandon.russell700@uky.edu

\end{document}